\documentclass[a4paper]{article}
\usepackage[utf8]{inputenc}
\usepackage[T1]{fontenc}      
\usepackage[english]{babel}
\usepackage{textcomp}
\usepackage{lmodern}
\usepackage{mathtools}
\usepackage{amssymb,amsfonts}
\usepackage{amsthm}
\usepackage[colorlinks,breaklinks,bookmarks,plainpages=false,unicode=true]{hyperref}%
\usepackage{datetime} 
\usepackage{subcaption}
\usepackage{tikz}
\usetikzlibrary{external,graphs}
\tikzset{myloop above/.style={loop, out=130, in = 50,min distance =8mm}}
\tikzset{myloop below/.style={loop, out=-130, in = -50,min distance =8mm}}
\tikzset{blacknode/.style={shape= circle, fill = black, inner sep = 0pt, outer sep = 0pt, minimum size = 3pt,draw}}
\newtheorem{lem}{Lemma}[section]

\newtheorem{cor}[lem]{Corollary}

\newtheorem{thm}[lem]{Theorem}
\theoremstyle{definition}
\newtheorem{defn}[lem]{Definition}
\theoremstyle{remark}

\DeclareMathOperator\Cayley{Cayl}
\DeclareMathOperator\Sch{Sch}
\DeclareMathOperator\stab{Stab}
\DeclareMathOperator\ab{ab}
\DeclarePairedDelimiter\abs{\lvert}{\rvert}
\newcommand*{\field}[1]{\mathbf{#1}}
\newcommand*{\Z}{\field{Z}}
\newcommand{\setst}[2]{\{#1\ |\ #2\}}
\title{Property~FW and wreath products of groups: a~simple approach using Schreier graphs}
\author{Paul-Henry Leemann\thanks{Supported by grant 200021\textunderscore188578 of the Swiss National Fund for Scientific Research.}, Grégoire Schneeberger}
\date{\today}
\begin{document}
\maketitle
\begin{abstract}
The group property~FW stands in-between the celebrated Kazdhan's property~(T) and Serre's property~FA. Among many characterizations, it might be defined, for finitely generated groups, as having all Schreier graphs one-ended.

It follows from the work of Y. Cornulier that a finitely generated wreath product $G\wr_XH$ has property~FW if and only if both $G$ and $H$ have property~FW and $X$ is finite.
The aim of this paper is to give an elementary, direct and explicit proof of this fact using Schreier graphs.
\end{abstract}
\section{Introduction}
Property~FW is a group property~that is (for discrete groups) a weakening of the celebrated Kazdhan's property~(T). It was introduced by Barnhill and Chatterji in \cite{Barnhill2008}. It is a fixed point property~for actions on wall spaces or, equivalently, on CAT(0) cube complexes. Therefore it stands between property~FH (fixed points on Hilbert spaces, equivalent to property~(T) for discrete groups) and property~FA (fixed points on \emph{arbres}\footnote{\emph{Arbres} is the french word for trees.}). It is known that all these properties are different, see  \cite{Cornulier2013} for examples of groups that distinguish them.

When working with group properties, it is natural to ask if they are stable under ``natural'' group operations. One such operation, of great use in geometric group theory, is the wreath product, that generalizes the direct product of two groups, see Definition~\ref{Def:WreathProd}.

In the context of properties defined by fixed points of actions, the first result concerning wreath products is due to Cherix, Martin and Valette and later refined by Neuhauser and concerns property~(T).
\begin{thm}[\cite{Cherix2004,Neuhauser2005a}] \label{T:Wreath_prop_T}
Let $G,H$ be two discrete groups with $G$ non-trivial and $X$ a set on which $H$ acts. The wreath product $G \wr_X H$ has property~(T) if and only if $G$ and $H$ have property~(T) and $X$ is finite.
\end{thm}
In \cite{Cornulier2013} Cornulier studied property~FW using \emph{cardinal definite functions} and while not explicitly stated in \cite{Cornulier2013}, the following result can be extracted from his work.
\begin{thm}\label{Thm:Main}
Let $G,H$ be two finitely generated groups with $G$ non-trivial and $X$ a set on which $H$ acts with finitely many orbits. The wreath product $G \wr_X H$ has property~FW if and only if $G$ and $H$ have property~FW and $X$ is finite.
\end{thm}
The aim of this note is to give an explicit and elementary proof of this fact using a characterization of property~FW via the number of ends of Schreier graphs, see Subsection \ref{Subsection:FW} for the relevant definitions.

At this point, the curious reader might have two questions. First, is it possible to extend Theorem \ref{Thm:Main} beyond the realms of finitely generated groups and of actions with finitely many orbits? And secondly, is there a link between Theorems \ref{T:Wreath_prop_T} and \ref{Thm:Main}?
In both cases, the answer is yes.

This is the subject of the more technical paper \cite{LS2021}, which gives an unified proof of Theorems \ref{T:Wreath_prop_T} and \ref{Thm:Main} as well as of similar results for the Bergman's property~and more.

\paragraph{Organization of the paper}
The next section contains all the definitions as well as some examples, while Section \ref{Section:Proof} is devoted to the proof of Theorem~\ref{Thm:Main} and some related results.
\section{Definitions and examples}

This section contains all the definitions, some standard but useful preliminary results as well as some examples.

\subsection{Ends of Schreier graphs and property~FW}
\label{Subsection:FW}
In what follows, we will always assume that generating sets of groups are \emph{symmetric}, that is we will look at $S\subset G$ such that $s\in S$ if and only if $s^{-1}\in S$.
Our graphs will be undirected and we will sometimes identify a graph with its vertex set.
%
%
%
%
%
%

Let $G$ be a group with symmetric generating set $S$, and let $X$ be a non-empty set endowed with a left action $G\curvearrowright X$.
The corresponding (left) \emph{orbital graph} $\Sch_{\mathcal O}(G,X;S)$ is the graph with vertex set $X$ and with an edge between $x$ and $y$ for every $s$ in $S$ such that $s.x=y$.
\begin{defn}
Let $G$ be a group, $H$ a subgroup of $G$ and $S$ a symmetric generating set. The \emph{(left) Schreier graph} $\Sch(G,H;S)$ is the graph with vertices the left cosets $gH=\setst{gh}{h\in H}$ and for every $\{s,s^{-1}\}\subset S$ an edge between $gH$ and $sgH$.
\end{defn}
As the notation suggests, these two definitions are two faces of the same coin. More precisely, Schreier graphs of $G=\langle S\rangle$ are exactly the orbital graphs for transitive actions (equivalently: the connected components of orbital graphs) of $G=\langle S\rangle$.
The correspondance is given by $H\mapsto G/H$ and $x_0\in X\mapsto\stab_G(x_0)$, where $x_0$ is an arbitrary element of $X$.

Schreier graphs are generalizations of the well-known \emph{Cayley graphs}, with $\Cayley(G;S)=\Sch(G,\{1\};S)$, see Figures~\ref{Figure:CayleyOfZ} and~\ref{Figure:CayleyOfZ2} for some examples.

If $S$ and $T$ are two generating sets of $G$, the graphs $\Sch(G,H;S)$ and $\Sch(G,H;T)$ does not need to be isomorphic. However, if both $S$ and $T$ are finite, then $\Sch(G,H;S)$ and $\Sch(G,H;T)$ are quasi-isometric, see~\cite[IV.B.21.iii]{DelaHarpe2000} for a proof and Figure~\ref{Figure:CayleyOfZ} for an example. The only fact we will use about quasi-isometries is that they preserve ``large-scale properties'' of the graph, as for example the number of ends.
Observe that the requirement that both $S$ and $T$ are finite is crucial for the existence of a quasi-isometry between the corresponding Schreier graphs. Indeed, for every group  $\Cayley(G;G\setminus\{1\})=\Sch(G,\{1\};G\setminus\{1\})$ is always a complete graph on $\abs{G}$ vertices, in particular  $\Cayley(\Z;\{1\})$ and $\Cayley(\Z;\Z\setminus\{0\})$ are not quasi-isometric.

\begin{figure}[htbp]\centering
\includegraphics{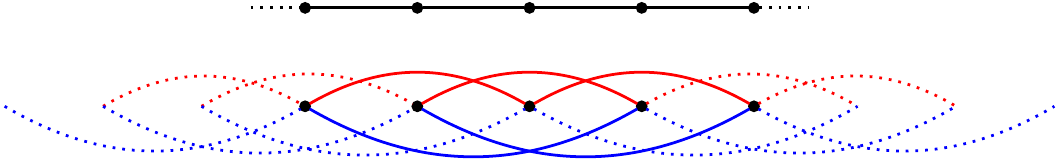}
\caption{Fragments of two Cayley graphs of $\Z$ ($2$ ends), for the standard generating set $\{\pm1\}$ and for the generating set $\{\textcolor{red}{\pm2},\textcolor{blue}{\pm3}\}$.}
\label{Figure:CayleyOfZ}
\end{figure}
\begin{figure}[htbp]\centering
\begin{subfigure}{0.5\textwidth}
\centering
\includegraphics{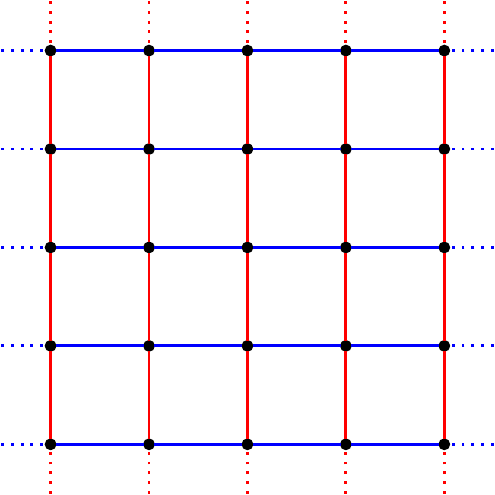}
\end{subfigure}%
\begin{subfigure}{0.5\textwidth}
\centering
\scalebox{0.8}{\includegraphics{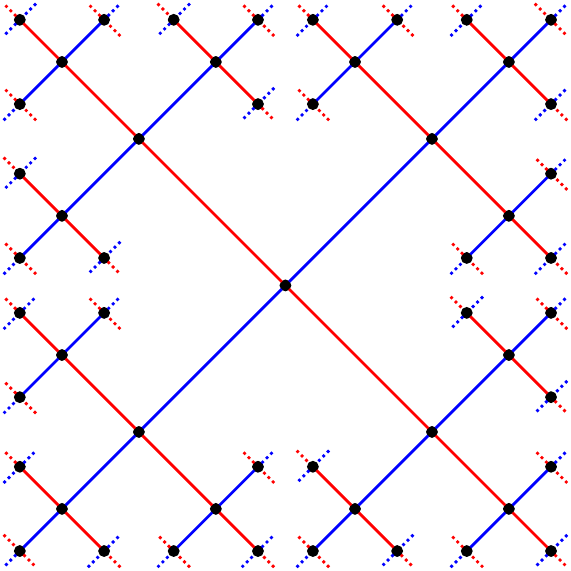}}
\end{subfigure}
\caption{Fragments of the Cayley graphs of $\Z^2$ ($1$ end) on the left and of $F_2$ (infinitely many ends) on the right; with standard generating sets.}
\label{Figure:CayleyOfZ2}
\end{figure}
\begin{figure}[htbp]\centering
\includegraphics{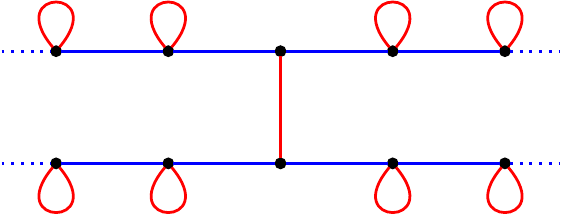}
\caption{A fragment of a Schreier graph (with $4$ ends) of the free group $F_2=\langle \textcolor{red}{x^{\pm1}},\textcolor{blue}{y^{\pm1}}\rangle$ for the subgroup $H=\setst{x^2,y^nxy^{-n},xy^nxy^{-n}x^{-1}}{n\in\Z\setminus\{0\}}$.}
\label{Figure:SchreierOfF2}
\end{figure}

Let $\Gamma$ be a graph and $K$ a finite subset of vertices. The graph $\Gamma\setminus K$ is the subgraph of $\Gamma$ obtained by deleting all vertices in $K$ and all edges containing them. This graph is not necessarily connected.
\begin{defn}\label{Def:Ends}
Let $\Gamma$ be a graph. The \emph{number of ends} of $\Gamma$ is the supremum, taken over all finite $K$, of the number of infinite connected components of $\Gamma\setminus K$.
\end{defn}
There exists other characterizations of the number of ends in graphs, 
see~\cite{MR1967888} and the references therein, but Definition~\ref{Def:Ends} is the one that best suits our purpose.
A locally finite graph (i.e. such that every vertex has finite degree) is finite if and only if it has $0$ end.

An important fact about the number of ends of a graph, is that it is an invariant of quasi-isometry, see~\cite{MR1213151}. In particular, if $G$ is a finitely generated group it is possible to speak about the number of ends of the Schreier graph $\Sch(G,H;S)$ without specifying a particular finite generating set $S$.
By a celebrated result of Hopf~\cite{MR10267}, the number of ends of a Cayley graph of a finitely generated group can only be $0$, $1$, $2$ or infinite (in which case it is uncountable), see Figures~\ref{Figure:CayleyOfZ} and~\ref{Figure:CayleyOfZ2} for some examples.
On the other hand, Schreier graphs may have any number of ends in $\mathbf N\cup\{\infty\}$, see Figure~\ref{Figure:SchreierOfF2} for an example of a graph with $4$ ends.
In fact, every regular graph of even degree is isomorphic to a Schreier graph,~\cite{MR0450121,MR1358635}.

We are now finally able to introduce property~FW. Instead of giving the original definition in terms of actions on wall spaces, we will use an equivalent one for finitely generated groups, which essentially follows from~\cite{MR1347406}, see~\cite{Cornulier2013} for a direct proof.
\begin{defn}
A finitely generated group $G$ has \emph{property~FW} if all its Schreier graphs have at most one end.
\end{defn}
It directly follows from the definition that all finite groups have property~FW, but that $\Z$ does not have it.
In fact, if $G$ is a finitely generated group with an homomorphism onto $\Z$, then it does not have FW. Indeed, in this case $G\cong H\rtimes \Z$  for some $H$ and the Schreier graph $\Sch(G,H;S)$ is isomorphic to a Cayley graph of $\Z\cong G/H$ and hence has $2$ ends.

Property~FW admits many distinct characterizations that allow to define it for groups that are non-necessarily finitely generated and even for topological groups. We refer the reader to~\cite{Cornulier2013} for a survey of these characterizations.

\subsection{Wreath products}
Let $X$ be a set and $G$ a group. We view
$\bigoplus_XG$ as the set of functions from $X$ to $G$ with finite support:
\[
	\bigoplus_XG=\setst{\varphi\colon X\to G}{\varphi(x)=1 \textnormal{ for all but finitely many }x}.
\]
This is naturally a group, where multiplication is taken pointwise.

If $H$ is a group acting on $X$, then it naturally acts on $\bigoplus_XG$
by $(h.\varphi)(x)=\varphi(h^{-1}.x)$.
This leads to the following standard definition.
\begin{defn}\label{Def:WreathProd}
Let $G$ and $H$ be groups and $X$ be a set on which $H$ acts.
The \emph{(retricted) wreath product} $G\wr_XH$ is the group $(\bigoplus_XG)\rtimes H$.
\end{defn}
A prominent  source of examples of wreath products is the ones of the form $G\wr_HH$, where $H$ acts on itself by left multiplication.
In particular, the group $(\Z/2\Z)\wr_\Z\Z$ is well-known under the name of the \emph{lamplighter group}.
The direct product $G\times H$ corresponds to wreath product over a singleton $G\wr_{\{*\}}H$.

Let $S$ be a generating set of $G$ and $T$ a generating set of $H$.
Let $\{x_i\}_{i\in I}$ be a choice of a representative in each $H$-orbit.
Finally, let $\delta_x^s$ be the element of $\bigoplus_XG$ defined by $\delta_x^s(x)=s$ and $\delta_x^s(y)=1_G$ if $y\neq x$ and let $\mathbf 1$ be the constant function with value $1_G$.
It is a direct computation that
\[
	\setst{(\delta_{x_i}^s,1_H)}{s \in S,i\in I} \cup \setst{(\mathbf 1,t)}{t \in T}
\]
is a generating set for $G\wr_XH$.

On the other hand, if $\setst{(\varphi_i,h_i)}{i\in I}$ is a generating set of $G\wr_XH$, then $\setst{h_i}{i\in I}$ is a generating set of $H$ while $\setst{\varphi_i(x)}{i\in I,x\in X}$ is a generating set of $G$.
Observe that since the $\varphi_i$ take only finitely many values, if $I$ is finite, so is $\setst{\varphi_i(x)}{i\in I,x\in X}$.
We hence recover the following characterization of the finite generation of $G\wr_XH$.

\begin{lem}
The group $G\wr_XH$ is finitely generated if and only if both $G$ and $H$ are finitely generated and $H$ acts on $X$ with finitely many orbits.
\end{lem}
\begin{proof}
If $G\wr_XH$ is finitely generated, so is its abelianization $(G\wr_XH)^{\ab}\cong (\bigoplus_{X/H}G^{\ab})\times H^{\ab}$, which implies that the orbit set $X/H$ is finite. The other implications directly follow from the above discussion on generating sets.
\end{proof}
%
%
%
%
Using the above lemma, we could reformulate Theorem~\ref{Thm:Main} in the following way: \textit{Let $G$, $H$ be two groups with $G$ non-trivial and $X$ a set on which $H$ acts. Suppose that all three of $G$, $H$ and $G\wr_XH$ are finitely generated. Then the wreath product $G\wr_XH$ has property~FW if and only if $G$ and $H$ have property~FW and $X$ is finite.}
While this formulation is formally equivalent to Theorem~\ref{Thm:Main}, it  hints the fact that the finite generation hypothesis on $G$, $H$ and $G\wr_XH$ are not necessary, but artefacts of using Schreier graphs in the proof. Indeed, the result remains true without these hypothesis, see \cite[Propositions 5.B.3 \& 5.B.4]{Cornulier2013} and \cite[Theorem A]{LS2021}.

\section{Proof of the main result}
\label{Section:Proof}

This section is devoted to the proof of Theorem~\ref{Thm:Main}. This proof is split into two parts: Lemma \ref{Lemma:Semidirect_ends} and its Corollary \ref{Cor:Wreath_ends}, and Lemma \ref{Lem:Wreath_groups_ends}.

We begin by an easy result on quotients.
\begin{lem}\label{Lem:Quotient}
Let $G$ be a finitely generated group and $H$ a quotient of $G$. If $G$ has FW, then so does $H$.
\end{lem}
\begin{proof}
First, remark that if $G$ is generated by a finite symmetric set $S$, the group $H$ is generated by the projection of $S$ that we will also denote by $S$.
Moreover, any generating set of $H$ can be obtained in such a way.

Let $K$ be any subgroup of $H$ and $L$ its preimage in $G$.
As $G/L \cong H/K$, forgetting about loops and multiedges, the Schreier graphs $\Sch(G,L;S)$ and $\Sch(H,K;S)$ are isomorphic.
By assumption on $G$, the graph $\Sch(G,L;S)$ has at most one end.
As adding loops or doubling edges do not change the number of ends, $\Sch(H,K;S)$ has also at most one end, and therefore $H$ has property~FW.
\end{proof}
We also have the following lemma on semi-direct products:
\begin{lem}\label{Lemma:Semidirect_ends}
Let $N$ and $H$ be two finitely generated groups and $N\rtimes H$ a semi-direct product.
Then
\begin{enumerate}
\item If $N\rtimes H$ has property~FW, then so does $H$,
\item If both $N$ and $H$ have property~FW, then $N\rtimes H$ also has property~FW.
\end{enumerate}
\end{lem}
\begin{proof}
The first part follows from Lemma~\ref{Lem:Quotient}.

For the second part, let $S$, respectively $T$, denote a finite generating set of $N$, respectively $H$.
Then the group $G\coloneqq N\rtimes H$ is finitely generated by $U=(S\times\{1\}) \cup(\{1\}\times T)$.

Suppose that both $N$ and $H$ have property~FW. We want to show that every Schreier graph of $G$ has at most one end. If they are all finite, then there is nothing to prove (and $G$ is finite). So let $\Gamma$ be an infinite Schreier graph of $G$ with respect to the generating set $U$. The groups $N$ and $H$ act on the vertices of $\Gamma$ by restriction of the action of $G$. That is, $n.x = (n,1).x$ and $h.x = (1,h).x$ for every vertex $x$ of $\Gamma$.
For each vertex $x$ we define $\Gamma_x^H$ (and respectively $\Gamma_x^N$) as the Schreier graph obtained from the action of $H$ (respectively $N$) on the $H$-orbit (respectively $N$-orbit) of $x$. These are subgraphs of $\Gamma$. As $N$ and $H$ have property~FW, the graphs $\Gamma_x^H$ and $\Gamma_x^N$ are either finite or one-ended; and we will prove that this implies that $\Gamma$ has exactly one end.

Let $K$ be a finite set of vertices of $\Gamma$.
If $x$ is in $K$ and $\Gamma_x^H$ is finite, add all vertices of $\Gamma_x^H$ to $K$.
By doing so for every $x$ in $K$, we obtain a new finite set $K' \supset K$ of vertices of $\Gamma$.
We will show that $\Gamma\setminus K'$ has only one infinite connected component.
By definition of $K'$, if $x$ is not in $K'$, then either $\Gamma_x^H$ has one end or $\Gamma_x^H$ does not contain vertices of $K'$.

Let $x$ and $y$ be two vertices, each of them lying in some infinite connected component of $\Gamma\setminus K'$.
We will construct a path from $x$ to $y$ in $\Gamma\setminus K'$ as a concatenation of three smaller paths, see Figure~\ref{Figure:PathSemiDirect}, as follows.
First, a path in $\Gamma_x^H\setminus K'$ from $x$ to some $z$, then a path in $\Gamma_z^N\setminus K'$ from $z$ to some $z'\in (\Gamma_z^N\cap \Gamma_y^H)\setminus K'$, and finally a path in $\Gamma_y^H\setminus K'$ from $z'$ to $y$.
In order to finish the proof, it remains to exhibit elements $z$ and $z'$ and the three desired paths.
\begin{figure}[htbp]\centering
\includegraphics{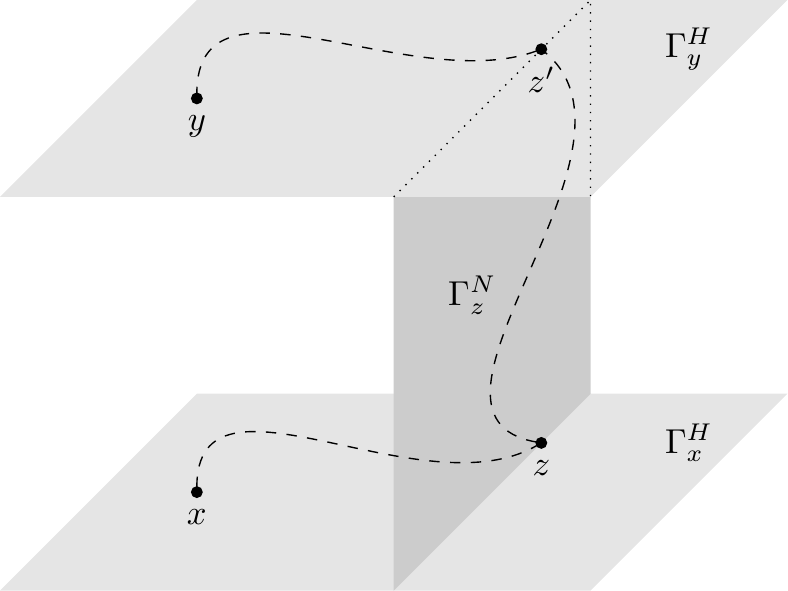}
\caption{The path between $x$ and $y$.}
\label{Figure:PathSemiDirect}
\end{figure}

The action of $G$ on $\Gamma$ being transitive, there exists an element $(n_0,h_0)$ of $N \rtimes H$ such that $(n_0,h_0).x = y$.
Since $K'$ is finite, the set $\Gamma_x^H\setminus K'$ is infinite.
Moreover, there is infinitely many $z$ in $\Gamma_x^H\setminus K'$ such that either $\Gamma_z^N$ is one-ended or $\Gamma_z^N$ does not intersect $K'$.
For such a $z$ there exists $h$ such that $(1,h).x=z$.
Now, the vertex $z'\coloneqq(hh_0^{-1}.n_0,h).x$ is both equal to $(hh_0^{-1}.n_0,1)(1,h_0).x=(hh_0^{-1}.n_0,1).z$ and to $(1,hh_0^{-1})(n_0,h_0).x=(1,hh_0^{-1}).y$. That is, $z'$ is in $\Gamma_z^N\cap \Gamma_y^H$.
A direct computation shows us that the map $z\mapsto z'$ is injective: $z_1'=z_2'$ if and only if $z_1=z_2$.
Since $K'$ is finite, there are only finitely many $z'$ in $K'$ and hence there are infinitely many $z\in \Gamma_x^H$ such that both $z$ and $z'$ are not in $K'$ and either $\Gamma_z^N$ is one-ended or $\Gamma_z^N$ does not intersect $K'$.

In order to finish the proof, observe that the three graphs $\Gamma_x^H$, $\Gamma_y^H$ and $\Gamma_z^N$ are all either one-ended or do not intersect $K'$.
Therefore, there is a path in $\Gamma_x^H\setminus K'$ from $x$ to $z$ as desired, and similarly for the paths from $z$ to $z'$ and $z'$ to $y$.
We have just proved that for any finite $K$ the graph $\Gamma\setminus K$ has only one infinite connected component and therefore that $\Gamma$ is one-ended.
\end{proof}
%
%
We have the following result on direct products that can be obtained as a corollary of Lemma~\ref{Lemma:Semidirect_ends}. It is also possible to give a short proof of it using Lemma~\ref{Lem:Quotient} and a direct argument; details are let to the reader.
\begin{cor}
Let $G$ and $H$ be two finitely generated groups. Then $G\times H$ has property~FW if and only if both $G$ and $H$ have property~FW.
\end{cor}
By iterating Lemma \ref{Lemma:Semidirect_ends}, we obtain
\begin{cor}\label{Cor:Wreath_ends}
Let $G$ and $H$ be two finitely generated groups and $X$ a set on which $H$ acts with finitely many orbits. Then,
\begin{enumerate}
\item
If $G\wr_X H$ has property~FW, then so does $H$,
\item
If both $G$ and $H$ have property~FW and $X$ is finite, then $G\wr_X H$ has property~FW.
\end{enumerate}
\end{cor}
The following Lemma finishes the proof of Theorem~\ref{Thm:Main}.
\begin{lem}\label{Lem:Wreath_groups_ends}
Let $G$ and $H$ be two finitely generated groups with $G$ non-trivial and such that $H$ acts on some set $X$ with finitely many orbits.
If $G\wr_XH$ has property~FW, then
\begin{enumerate}
\item $G$ has property~FW,
\item $X$ is finite.
\end{enumerate}
\end{lem}
\begin{proof}
Let us fix some finite generating sets $S$ and $T$ of $G$ and $H$ and let
\[
	U\coloneqq\setst{(\delta_x^s,1_H)}{s \in S} \cup \setst{(\mathbf 1,t)}{t \in T}
\]
be the standard generating set of $G\wr_XH$.
\paragraph{If $G\wr_XH$ has property~FW, then $X$ is infinite.}
Since $H$ acts on $X$ with finitely many orbits, it is enough to show that each orbit is finite.
So let $X'$ be an orbit and $x_0$ be an arbitrary vertex of $X'$.
The group $G$ acting on itself by left multiplications, we have the so-called \emph{imprimitive action} of the wreath-product $G\wr_XH$ on $Y\coloneqq G\times X'$:
\[
	(\varphi,h). (g,x) \coloneqq (\varphi(h.x)g, h.x).
\]
Since both $G\curvearrowright G$ and $H\curvearrowright X'$ are transitive, the action $G\wr_XH\curvearrowright Y$ is also transitive.
Therefore, the orbital Schreier graph of $G\wr_XH\curvearrowright Y$ is isomorphic to the Schreier graph $\Gamma\coloneqq\Sch\left(G\wr_XH, \stab((1_G,x_0)\right), U)$. We decompose this graph into leaves of the form $Y_g = \{ g \} \times X'$. There are two types of edges in $\Gamma$, which are coming from the two sets of generators, see Figure~\ref{Figure:Leaves}. The first ones, of the form $(\mathbf 1,t)$, give us on each leaf a copy of the orbital Schreier graph of $H \curvearrowright X'$. Indeed,
\[
	(\mathbf 1,t).(g,x) = (g, t.x).
\]
The second ones, of the form $(\delta_{x_0}^s,1)$, give us loops everywhere except on vertices of the form $(g,x_0)$. By direct computation, we see that the vertices $(g,x_0)$ and $(sg,x_0)$ connect the leaves $Y_g$ and $Y_{sg}$,
\[
	(\delta_{x_0}^s,1).(g,x) =
		\begin{cases}
		(g,x) & \textnormal{if }x \neq x_0, \\
		(sg, x) & \textnormal{if }x = x_0.
		\end{cases}
\]
\begin{figure}[htbp]\centering
\scalebox{0.85}{\includegraphics{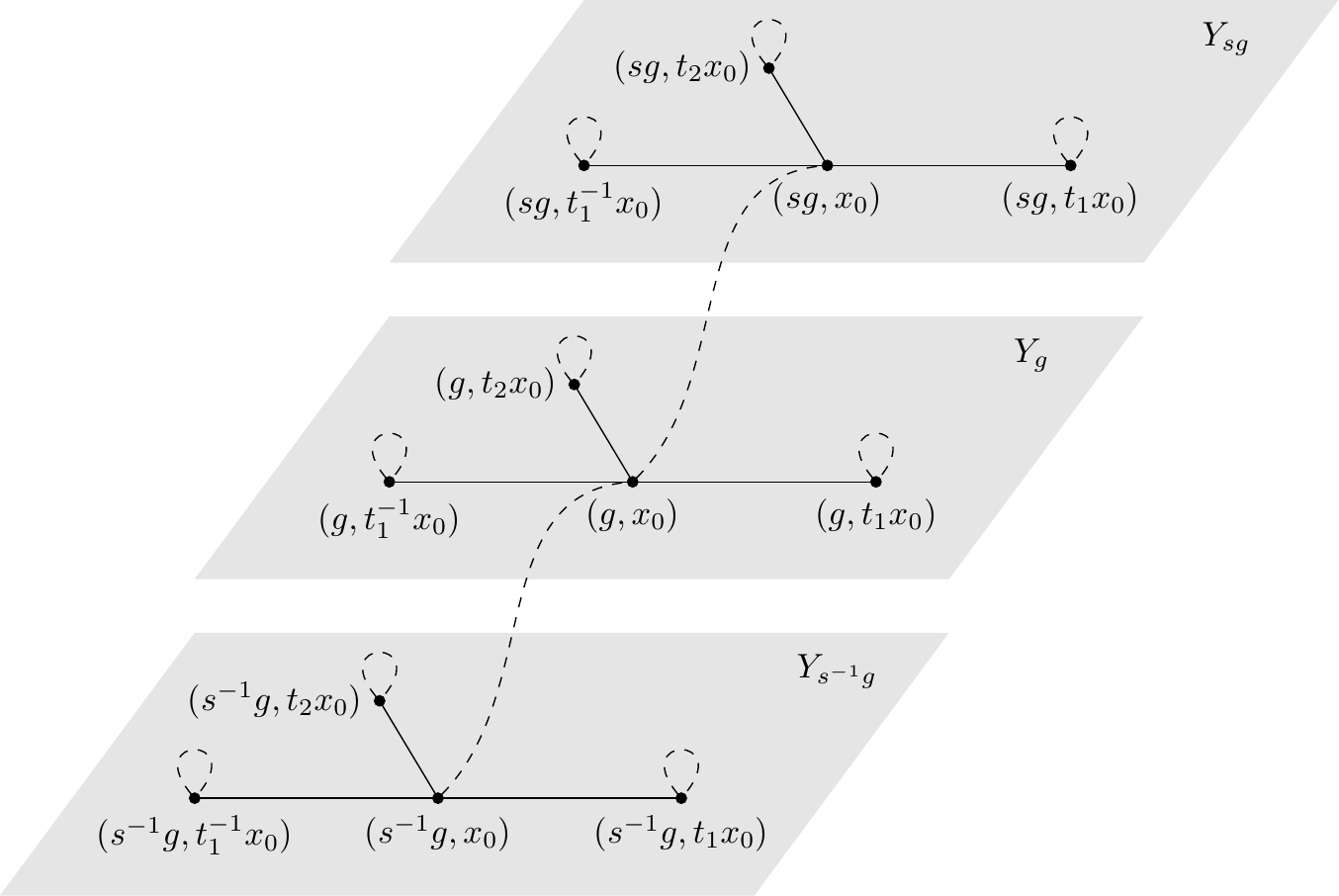}}
\caption{The leaf structure of the orbital Schreier graph of $G\wr_XH \curvearrowright Y$. Plain edges correspond to generators of the form $(\mathbf 1,t)$ while dotted edges correspond to generators of the form $(\delta_{x_0}^s,1)$.}
\label{Figure:Leaves}
\end{figure}

If we remove a vertex $(g,x_0)$ we disconnect the leaf $Y_g$ from the rest of $\Gamma$.
Since $\Gamma$ has at most one end and there is $\abs G\geq 2$ leaves, we deduce that all leaves $Y_g$ are finite, and hence that $X'$ itself is finite.
%
%
%
\paragraph{The group $G$ has property~FW.}
Let $K$ be any subgroup of $G$. We will show that $\Sch(G,K,S)$ has at most  one end.
Let $x_0$ be any point of $X$ and $X'$ be its orbit under the action of $H$.
We have the imprimitive action of $G\wr_XH$ on $G/K\times X$, which restricts to an action on $G/K\times X'$:
\[
	(\varphi,h).(gK,x) = (\varphi(h.x) gK, h.x).
\]
As above, the action is transitive and the orbital Schreier graph of this action is isomorphic to a Schreier graph $\Gamma$ of $G\wr_XH$. We decompose this graph into leaves in the same way.
Now observe that $\Sch(G,K,S)$ is isomorphic to the subgraph $\Delta$ of $\Gamma$ consisting of vertices $\setst{(g,x_0)}{g\in G}$ and edges $\setst{(\delta_{x_0}^s,1)}{s\in S}$.
Due to the leaves structure of $\Gamma$, the number of ends of $\Delta$ is bounded above by the number of ends of $\Gamma$, and hence is at most one.
We conclude that $\Sch(G,K,S)$ too has at most one end.
\end{proof}
\paragraph{Acknowledgment}
The authors are thankful to A. Genevois, T. Nagnibeda and A. Valette for helpful comments on a previous version of this note, and to the Swiss National Fund for Scientific Research for its support.
The author also thank the anonymous referee for their valuable comments.

\enddocument